\documentclass{amsart}
\usepackage[mathcal]{eucal}
\usepackage{amssymb, graphicx, labelfig
}

\newtheorem{thm}{Theorem}

\newtheorem{lem}[thm]{Lemma}

\theoremstyle{remark}
\newtheorem{rem}[thm]{Remark}

\theoremstyle{definition}

\newcommand{\col}{\kern -3pt :}
\newcommand{\C}{\mathbb C}

\newcommand{\N}{\mathbb N}

\newcommand{\Id}{\mathrm{Id}}
\newcommand{\End}{\mathrm{End}}

\newcommand{\SSS}{\mathcal S^A}

\newcommand{\RR}{\mathcal R}

\newcommand{\SL}{\mathrm{SL}_2}

\newcommand{\Tr}{\mathrm{Tr}}
\newcommand{\E}{\mathrm{e}}
\newcommand{\I}{\mathrm{i}}

\newcommand{\ZWRT}{Z^A_{\mathrm{WRT}}}

\newcommand{\db}{/\kern -4pt/}

\renewcommand{\leq}{\leqslant}
\renewcommand{\geq}{\geqslant}
\renewcommand{\phi}{\varphi}
\renewcommand{\epsilon}{\varepsilon}

\title[The Witten-Reshetikhin-Turaev representation]
{The Witten-Reshetikhin-Turaev\\ representation of the \\Kauffman skein algebra}

\author{Francis Bonahon}
\address {Department
of Mathematics,  University of
Southern California, Los Angeles
CA~90089-2532, U.S.A.}
\email{fbonahon@math.usc.edu}

\author{Helen Wong}
\address{Department
of Mathematics, Carleton College, Northfield MN 55057, U.S.A.}
\email{hwong@carleton.edu}
\thanks{This research was partially supported by grants DMS-0604866, DMS-1105402 and DMS-1105692  from the National Science Foundation, and by a mentoring grant from the Association for Women in Mathematics.}
\date{\today}


\begin{document}

\commby{XXX}

\begin{abstract}
For $A$ a primitive $2N$--root of unity with $N$ odd, the Witten-Reshetikhin-Turaev topological quantum field  theory provides a representation of the Kauffman skein algebra of a closed surface. We show that this representation is irreducible, and we compute its classical shadow in the sense of~\cite{BonWonSkeinReps1}. 
\end{abstract}
\maketitle

The discovery of the Jones polynomial \cite{Jones}, and the simplification of its construction by Kauffman \cite{Kauffman}, were quickly followed by two originally unrelated developments. The first one was the introduction of the Kauffman skein module of an oriented 3--manifold  by Turaev \cite{TuraevSkein} and Przytycki \cite{PrzytyckiSkein}. A special case leads to the Kauffman skein algebra $\SSS(S)$ of an oriented surface $S$ which, when $S$ is connected,  was later interpreted as a quantization of the character  variety $\mathcal R_{\SL(\C)}(S)$, consisting of the characters of all group homomorphisms $\pi_1(S) \to \SL(\C)$ \cite{TuraevPoisson, BFKBKauffmanSkeinModule, BFKBKauffmanSkeinObservable, PrzSik}. 

Another development was Witten's interpretation \cite{Witten} of the Jones polynomial within the framework of  a topological quantum field theory. This topological quantum field theory point of view was formalized in mathematical terms by Reshetikhin-Turaev \cite{ReshetikhinTuraev1, ReshetikhinTuraev2}. In particular, the Witten-Reshetikhin-Turaev topological quantum field theory leads, for every primitive $2N$--root of unity $A$, to a representation $\rho \colon \SSS(S) \to \End(V_S)$ of the skein algebra corresponding to this parameter $A$.  

The first  result of this article is the following.

\begin{thm}
\label{thm:WRTirreducibleIntro}
Let $S$ be a connected closed oriented surface. For every primitive $2N$--root of unity $A$, the Witten-Reshetikhin-Turaev  representation $\rho \colon \SSS(S) \to \End(V_S)$ is irreducible. 
\end{thm}

This result can be compared with Roberts's proof \cite{RobertsIrred} that, when $N=2p$ with $p$ prime, the action of the mapping class group of $S$ on the  Witten-Reshetikhin-Turaev  space $V_S$ is irreducible. In this special case, Theorem~\ref{thm:WRTirreducibleIntro} can actually be deduced from some of the proofs of \cite{RobertsIrred}.

Our interest in the irreducibility of the Witten-Reshetikhin-Turaev  representation is motivated by  \cite{BonWonJacoFest, BonWonSkeinReps1, BonWonSkeinReps2, BonWonSkeinReps3}, where we  initiated the systematic study of finite-dimensional irreducible representations of the skein algebra $\SSS(S)$. In particular, when $A$ is a $2N$--root of unity with $N$ odd, we associate to such an irreducible  representation $\rho \colon \SSS(S) \to \End(V)$ an element $r_\rho$ of the character variety $\mathcal R_{\SL(\C)}(S)$; see Theorem~\ref{thm:ClassicShadow} for a precise statement. If we regard $\SSS(S)$ as a quantization of $\mathcal R_{\SL(\C)}(S)$ and a representation $\rho$ as a  point of this quantization,  the character $r_\rho\in \mathcal R_{\SL(\C)}(S)$ is the \emph{classical shadow} of $\rho$. 
With irreducibility of the Witten-Reshetikhin-Turaev  representation established, we thus may  inquire about the classical shadow of the best known representation of the skein algebra. 

The Witten-Reshetikhin-Turaev topological quantum field  theory, and the associated representation of $\SSS(S)$, have slightly different features according to whether $N$ is even or odd, respectively known as the $\mathrm{SU}_2$ and $\mathrm{SO}_3$ cases. The classical shadow of a representation is defined only when $N$ is odd.

\begin{thm}
\label{thm:WRTshadowIntro}
When $A$ is a primitive $2N$--root of unity with $N$ odd, the classical shadow of the Witten-Reshetikhin-Turaev  representation $\rho \colon \SSS(S) \to \End(V_S)$ is the trivial character $\iota \in \mathcal R_{\SL(\C)}(S)$, represented by the trivial homomorphism $\pi_1(S) \to \SL(\C)$. 
\end{thm}

The proof of Theorem~\ref{thm:WRTshadowIntro} is relatively simple, although it uses deep connections in the quantum theory with both types of  Chebyshev polynomials. The Chebyshev polynomial  of the first type, $T_N(x)$, is used to define the classical shadow of a representation of $\SSS(S)$. On the other hand, the second type of Chebyshev polynomials, $S_n(x)$,  classically plays an important r\^ole in the representation theory of the quantum group $\mathrm U_q(\mathfrak{sl}_2)$ and in the  Witten-Reshetikhin-Turaev topological quantum field theory.  We are ultimately led to Theorem~\ref{thm:WRTshadowIntro} by exploiting these relationships and making use of some elementary relations between the two types. 

In \cite{BonWonSkeinReps3} we construct, for every character $r\in\mathcal R_{\SL(\C)}(S)$, an irreducible representation $\rho \colon \SSS(S) \to \End(V_S)$ whose classical shadow $r_\rho$ is equal to $r$. In particular, this associates another irreducible representation $\rho_\iota$ to the trivial character $\iota$.   The construction of  \cite{BonWonSkeinReps3} is unfortunately not very explicit in the special case of the trivial character $\iota$; however, it appears that the representation that it provides has very different features from those of the Witten-Reshetikhin-Turaev representation.

\section{The Kauffman skein module}

The \emph{Kauffman skein module} $\SSS(M)$  of an oriented 3--dimensional manifold $M$ depends on a parameter $A=\E^{\pi\I \hbar}\in \C-\{0\}$, and is defined as follows: one first considers the vector space freely generated by  all isotopy classes of framed links in the thickened surface $S \times [0,1]$, and then one takes the quotient of this space by two relations:
\begin{itemize}
\item the first relation is the  \emph{skein relation} that
$
[L_1] = A^{-1} [L_0] + A [L_\infty]
$
whenever the three links $L_1$, $L_0$ and $L_\infty\subset S\times [0,1]$ differ only in a little ball where they are as represented on Figure~\ref{fig:SkeinRelation};

\item the second relation states that $[L\cup O] = -(A^2 +A^{-2})[L]$ whenever the knot $O$ is the boundary of a disk $D$ endowed with framing transverse to $D$, and the framed link $L$ is disjoint from the disk $D$.

\end{itemize}

\begin{figure}[htbp]

\SetLabels
( .5 * -.4 ) $L_0$ \\
( .1 * -.4 )  $L_1$\\
(  .9*  -.4) $L_\infty$ \\
\endSetLabels
\centerline{\AffixLabels{\includegraphics{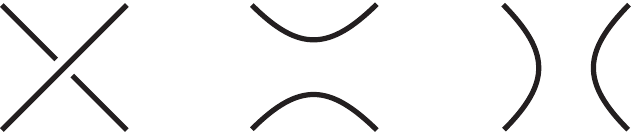}}}
\vskip 15pt
\caption{A Kauffman triple}
\label{fig:SkeinRelation}
\end{figure}

In the special case where $M=S\times [0,1]$ for an oriented surface $S$, we write $\SSS(S) = \SSS \bigl (S\times [0,1]\bigr)$ and we note that this module now comes with a natural multiplication. Indeed, if $[L_1]$, $[L_2]\in \SSS(S)$ are respectively represented by the framed links $L_1$, $L_2$, we can consider their superposition
$$
[L_1] \cdot [L_2] = [L_1' \cup L_2'] \in \SSS(S)
$$
 represented by the union of the framed link $L_1'\subset S\times[0, \frac12]$ obtained by rescaling $L_1 \subset S\times [0,1]$ and of the framed link $L_2'\subset S \times [\frac12, 1]$ obtained by rescaling $L_2 \subset S\times [0,1]$. This endows the skein module $\SSS(S)$ with the structure of an algebra, called the  \emph{Kauffman skein algebra} of the oriented surface $S$.

\section{The Witten-Reshetikhin-Turaev topological quantum field theory}
\label{sec:WRTtqft}

We briefly review a few fundamental properties of the Witten-Reshetikhin-Turaev   topological quantum field theory, and refer to
\cite{BHMV3man, BHMVTQFT, LickorishSkein3manInvs, LickorishBook, LickorishHandbook,TuraevBook} for  details and proofs.

The Witten-Reshetikhin-Turaev   topological quantum field theory $\ZWRT$ depends on the choice of a primitive $2N$--root of unity $A$, and is defined over the category $\mathcal C$ defined as follows: the objects of $\mathcal C$ are closed oriented surfaces $S$; the morphisms from $S_1$ to $S_2$ are pairs $(M,L)$ where $M$ is a compact oriented  3--manifold with $\partial M = (-S_1) \sqcup S_2$, where $L$ is a framed link in the interior of $M$, and where $M$ is endowed with a $p_1$--structure. The precise definition of a $p_1$--structure can be found in \cite[App.~B]{BHMVTQFT} but, for the purpose of the current article, it suffices to know that it captures  certain homotopic information on the tangent bundle of the manifolds considered.

In particular, $\ZWRT$ associates a finite-dimensional vector space $V_S= \ZWRT(S)$  to each closed oriented surface $S$, and a linear map $\ZWRT{(M,L)} \colon V_{S_1} \to V_{S_2}$ to each morphism $(M,L)$ as above. In addition, the vector space $V_\varnothing$ associated to the empty surface $\varnothing$ comes equipped with a canonical identification with $\C$; in other words, this vector space is 1--dimensional and contains a preferred basis element that we will denote by 1. 

The topological quantum field theory $\ZWRT$ satisfies many  properties, in particular those that characterize topological quantum field theories. Most of these features will play no direct  r\^ole in the current article. However the following fact, which is grounded in properties of the quantum group $\mathrm U_q(\mathfrak{sl}_2)$  underlying the construction of $\ZWRT$, is crucial for our purposes.

\begin{lem}
\label{lem:WRTmorphSkeinRel}
The linear maps $\ZWRT{(M,L)}$ associated to the morphisms of the category $\mathcal C$ satisfy the skein relation that, as linear maps $ V_{S_1} \to V_{S_2}$,
$$
\ZWRT{(M,L_1)} =   A^{-1} \ZWRT{(M,L_0)} + A \, \ZWRT{(M,L_\infty)} 
$$
whenever the framed links $L_1$, $L_0$ and $L_\infty$ form a Kauffman triple in the manifold $M$ with $\partial M = (-S_1) \sqcup S_2$, for a fixed $p_1$--structure on $M$. 

Also, 
$$
\ZWRT(M, L \cup O) = -(A^2 + A^{-2}) \ZWRT(M,L)
$$
whenever the knot $O$ is the boundary of a disk $D$ endowed with framing transverse to $D$, and the framed link $L$ is disjoint from the disk $D$.    \qed
\end{lem}

A first consequence of Lemma~\ref{lem:WRTmorphSkeinRel} is that the skein algebra $\SSS(S)$ acts on the space $V_S$, by considering the special case $M= S\times [0,1]$. Indeed, for any framed link $L \subset S \times[0,1]$, the pair $(S \times [0,1] ,L)$ can be seen as a morphism from $S$ to $S$, and therefore induces a linear map $\ZWRT(S\times [0,1], L) \colon V_S \to V_S$. 

\begin{lem} \label{Lemma:rho}
\label{lem:WRTrepExists}
There exists a unique algebra homomorphism
$$
\rho\colon \SSS(S) \to \End(V_S)
$$ 
such that, for every framed link $L$ in $S\times [0,1]$,
$$
\rho \bigl( [L] \bigr) = \ZWRT(S\times [0,1], L)
$$
when $S\times [0,1]$ is endowed with the product $p_1$--structure.
\end{lem}
\begin{proof}
Lemma~\ref{lem:WRTmorphSkeinRel} shows that the rule $L \mapsto \ZWRT(S\times [0,1], L)$ is compatible with the skein relation, and therefore induces a linear map $\rho\colon \SSS(S) \to \End(V_S)$.  The multiplication law of $\SSS(S)$ is defined by the superposition operation, which itself corresponds to the composition of morphisms $(S\times [0,1], L)$ in the category $\mathcal C$. It follows  that $\rho$ is an algebra homomorphism.
\end{proof}

This homomorphism 
$
\rho\colon \SSS(S) \to \End(V_S)
$
is the \emph{Witten-Reshetikhin-Turaev representation} of the skein algebra $\SSS(S)$. 

Another application of Lemma~\ref{lem:WRTmorphSkeinRel} will enable us to perform computations in the space $V_S$. 
Consider a 3-manifold $M$ with boundary $\partial M =S$, endowed with a $p_1$--structure. A framed link $L \subset M$ provides a morphism  $(M,L)$ from the empty surface $\varnothing$ to $S$. This provides a linear map $\ZWRT(M,L)$ from $V_\varnothing=\C$ to $V_S$, and in particular specifies an element $\ZWRT(M,L)(1) \in V_S$. As above, Lemma~\ref{lem:WRTmorphSkeinRel} shows that the map $L \mapsto \ZWRT(M,L)(1)$ defines a linear map $\Phi_M \colon \SSS(M) \to V_S$, from the skein module of $M$ to the space $V_S$. 

\begin{lem} \label{Lemma:Phi}
\label{lem:WRTspaceSpannedBySkeins}
If $M$ is an oriented  $3$--manifold with boundary $\partial M = S$, the linear map $\Phi_M \colon \SSS(M) \to V_S$ defined by $\Phi_M\bigl([L]\bigr)=\ZWRT(M,L)(1)$ is surjective. \qed
\end{lem}

We will use Lemma~\ref{lem:WRTspaceSpannedBySkeins} in the special case where $S$ is connected and where $M$ is a handlebody $H$ with boundary $\partial H = S$. Choose an identification $H\cong\Sigma \times [0,1]$ of this handlebody with the product of the interval with  a compact oriented surface  $\Sigma$ with  boundary. 
Select also a trivalent spine for $\Sigma$, namely a trivalent graph $\Gamma$ embedded in the interior of  $\Sigma$ such that $\Sigma$ deformation retracts to $\Gamma$.  We will use this data to describe a basis for~$V_S$. 

An \emph{$N$--admissible weight system} assigns to each edge $e$ of $\Gamma$ a non-negative  integer weight $w(e)$ such that the following conditions hold:
\begin{enumerate}

\item at every vertex of $\Gamma$, the weights  of the edges $e_1$, $e_2$, $e_3$ adjacent to this vertex satisfy the triangle inequalities $w(e_1) \leq w(e_2) +w(e_3) $, $w(e_2) \leq w(e_1) +w(e_3) $ and $w(e_3) \leq w(e_1) +w(e_2) $;

\item if $N$ is odd, the weight $w(e)$ of each edge $e$ is even and bounded by $N-2$; in addition, at each vertex, the sum of the weights of the adjacent edges is bounded by $2N-4$;

\item if $N$ is even, the weight $w(e)$ of each edge $e$ is bounded by $\frac N2 -2$; in addition, at each vertex, the sum of the weights of the adjacent edges is even and bounded by $N-4$;

\end{enumerate}
Let $\mathcal W_\Gamma$ denote the (finite) set of all $N$--admissible weight systems for $\Gamma$.

An $N$--admissible edge weight system $w\in \mathcal W_\Gamma$ specifies  an element $\beta_w$ of the skein module $\SSS(H)$, by replacing  each edge $e$ of $\Gamma$ weighted by $w(e) \leq N-2$ by a copy of the $w(e)$--th Jones-Wenzl  idempotent, represented by a box \SetLabels
\E( .5* .5) $w(e) $ \\
( * ) $ $ \\
( * ) $ $ \\
( * ) $ $ \\
( * ) $ $ \\
\endSetLabels
\raisebox{-5pt}{\AffixLabels{\includegraphics{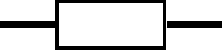}}}.
 The precise definition of the element $\beta_w\in \SSS(H)$ associated to the weight system $w\in \mathcal W_\Gamma$ can be found for instance in \cite{LickorishBook} or \cite{TuraevBook}, but we can give a flavor of the construction. For $a\geq 0$ such that $A^{4k}\neq 1$ for every $k$ with $0<k<a$ (namely, for $0\leq a <N$ is $N$ is odd, and for $0 \leq a<  \frac N2$ if $N$ is even), the  Jones-Wenzl  idempotent  \SetLabels
\E( .5* .5) $a $ \\
\endSetLabels
\raisebox{-5pt}{\AffixLabels{\includegraphics{JW.pdf}}} 
is a certain formal linear combination of families of disjoint arcs, each with $a$ strands emanating from each end of the box lying on $\Sigma$. For instance,

\begin{align*}\textstyle
\SetLabels
\E( .5* .5) 4 \\
\endSetLabels
\raisebox{-5pt}{\AffixLabels{\includegraphics{JW.pdf}}}
&=\textstyle
\raisebox{-7pt}{\includegraphics{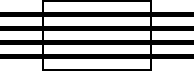}}
+
\frac{A^4+A^{-2}}{A^4+A^{-4}}\,
\raisebox{-7pt}{\includegraphics{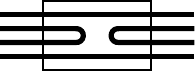}}
+
\frac{A^4+1+A^{-4}}{A^6+A^2+A^{-2}+A^{-6}}\,
\raisebox{14pt}{\rotatebox{180}{\includegraphics{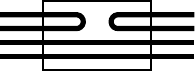}}}\\
&\quad\textstyle
+
\frac{A^4+1+A^{-4}}{A^6+A^2+A^{-2}+A^{-6}}\,
\raisebox{-7pt}{\includegraphics{JW41.pdf}}
+
\frac1{A^4+A^{-4}}\,
\raisebox{-7pt}{\includegraphics{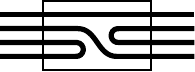}}\\
&\quad\textstyle
+
\frac1{A^4+A^{-4}}\,
\raisebox{-7pt}{\reflectbox{\includegraphics{JW43.pdf}}}
+
\frac1{A^4+A^{-4}}\,
\raisebox{14pt}{\rotatebox{180}{\includegraphics{JW43.pdf}}}
+
\frac1{A^4+A^{-4}}\,
\raisebox{14pt}{\rotatebox{180}{\reflectbox{\includegraphics{JW43.pdf}}}}\\
&\quad\textstyle
+
\frac1{A^6+A^2+A^{-2}+A^{-6}}\,
\raisebox{-7pt}{\includegraphics{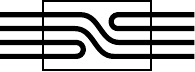}}
+
\frac1{A^6+A^2+A^{-2}+A^{-6}}\,
\raisebox{-7pt}{\reflectbox{\includegraphics{JW46.pdf}}}\\
&\quad\textstyle
+
\frac{A^2+A^{-2}}{(A^4+A^{-4})(A^4+1+A^{-4})}\,
\raisebox{-7pt}{\includegraphics{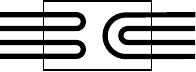}}
+
\frac{A^2+A^{-2}}{(A^4+A^{-4})(A^4+1+A^{-4})}\,
\raisebox{-7pt}{\reflectbox{\includegraphics{JW44.pdf}}}\\
&\quad\textstyle
+
\frac1{(A^4+A^{-4})(A^4+1+A^{-4})}\,
\raisebox{-7pt}{\includegraphics{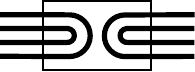}}
+
\frac{(A^2+A^{-2})^2}{(A^4+A^{-4})(A^4+1+A^{-4})}\,
\raisebox{-7pt}{\includegraphics{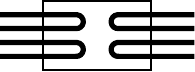}}
\end{align*}

\noindent These strands are then connected by disjoint arcs near the vertices of $\Gamma$, using the fact that at each vertex the $w(e)$  add up to an even number and satisfy the triangle inequalities. See Figure~\ref{fig:vertex}, where $w(e_1) =6$, $w(e_2)=4$ and $w(e_3) =4$. 

\begin{figure}[htbp]

\SetLabels
\E( .26* .5) $w(e_1) $ \\
( .84* .67) \rotatebox{60}{$w(e_2) $} \\
(.84 * .31)\rotatebox{-60}{$w(e_3) $} \\
( * ) $ $ \\
( * ) $ $ \\
\endSetLabels
\centerline{\AffixLabels{\includegraphics{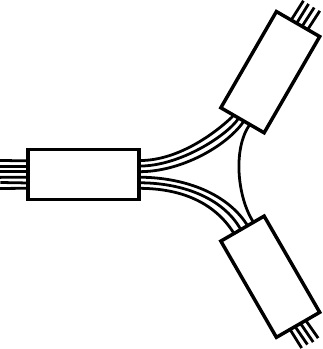}}}

\caption{}
\label{fig:vertex}
\end{figure}

For an $N$--admissible weight system $w\in \mathcal W_\Gamma$, let $\beta_w \in \SSS(H)$ be associated to $w$ as above, and let $b_w \in V_S$ be the image of $\beta_w$ under the map $\Phi_H \colon \SSS(H) \to V_S$ of Lemma~\ref{lem:WRTspaceSpannedBySkeins}. 

\begin{lem}
\label{lem:WRTbasis}
The subset $\mathcal B_\Gamma = \{ b_w \}_{w\in \mathcal W_\Gamma}$ is a basis for the vector space $V_S$.  \qed
\end{lem}

The weight system space $\mathcal W_\Gamma$ contains a special element $0$, assigning weight $0$ to each edge of $\Gamma$. This provides a prefered element $b_0\in V_S$, represented by the empty skein $[\varnothing ]\in \SSS(H)$. By definition, $b_0$ is the \emph{vacuum element} of $V_S$.

\section{The Witten-Reshetikhin-Turaev representation is irreducible}

I the rest of the article, $S$ will always denote a connected closed oriented surface. 
Consider the Witten-Reshetikhin-Turaev representation  $\rho \colon \SSS(S) \to \End(V_S)$ of Lemma~\ref{lem:WRTrepExists}.

\begin{thm}
\label{thm:WRTirred}
The Witten-Reshetikhin-Turaev representation  $\rho \colon \SSS(S) \to \End(V_S)$ is irreducible. 
\end{thm}

\begin{proof}
We will split the proof into several steps. 
Let $W \subset V_S$ be a non-trivial linear subspace that is invariant under the image $\rho \bigl( \SSS(S) \bigl)$. We want to show that $W$ is equal to the whole space $V_S$. 
For this, we will use a handlebody $H\cong \Sigma\times[0,1]$ bounding the surface $S$, a trivalent spine $\Gamma$ for the surface $\Sigma$, and the basis $\mathcal B_\Gamma = \{ b_w \}_{w\in \mathcal W_\Gamma}$ of Lemma~\ref{lem:WRTbasis}. 

By hypothesis, $W$ is non-trivial, and therefore contains a non-trivial element $\sum_{w\in \mathcal W_\Gamma} \alpha_w b_w$ with $\alpha_w \in \C$.
Our first step is   borrowed from  \cite{RobertsIrred} and \cite{RobertsSkeinMCG}.

\begin{lem}
\label{lem:InvSpaceContainsBasis}
Let $\Gamma$ be a trivalent spine for the surface $\Sigma$. 
If $\sum_{w\in \mathcal W_\Gamma} \alpha_w b_w$ is in the invariant subspace $W$, then every basis element $b_w \in \mathcal B_\Gamma$ with non-zero coefficient $\alpha_w\neq 0$ also belongs to $W$. 
\end{lem}
\begin{proof}
[Proof of Lemma~\ref{lem:InvSpaceContainsBasis}]
For every edge $e$ of $\Gamma$, there exists a disk $D_e \subset H$ such that $D_e \cap \partial H$ is equal to the boundary $\partial D_e$, and such that $D\cap \Gamma = D\cap e$ consists of a single point; this is an immediate consequence of the fact that  the handlebody $H$ deformation retracts to the graph $\Gamma$. Consider $[\partial D_e] \in \SSS(S)$ and its image $\rho \bigl( [ \partial D_e] \bigr) \in \End(V_S)$. For every basis element $b_w\in \mathcal B_\Gamma$ associated to the weight system $w\in \mathcal W_\Gamma$, a computation as in \cite[Lemma~14.2]{LickorishBook} shows that 
$$
\rho\bigl( [\partial D_e] \bigr) (b_w) = - (A^{2(w(e)+1)} + A^{-2(w(e)+1)}) b_w. 
$$

As $A$ is a primitive $2N$--root of unity, the numbers $A^{2(i+1)}+A^{-2(i+1)}$ are non-zero and distinct as $i$ ranges over all admissible $i$, that is, over all $i \in \{0,2,4,\dots, N-3\}$ when $N$ is odd and over all $i \in \{ 0, 1, 2, \ldots, \frac N2 -2\}$ when $N$ is even.

After these observations, the lemma is just a matter of elementary linear algebra. 
Consider an element $v=\sum_{w\in \mathcal W_\Gamma} \alpha_w b_w$   of $W$ such that $\alpha_{w_1}\neq 0$. We want to show that $b_{w_1}$ also belongs to $W$

If all the other coefficients $\alpha_w$ where $w\neq w_1$ are equal to 0, then $b_{w_1} = \frac 1 {\alpha_{w_1}} v\in W$ and we are done. 

Otherwise,  there exists another weight system $w_2\neq w_1$ with $\alpha_{w_2}\neq 0$. The fact that $w_2\neq w_1$ in $\mathcal W_\Gamma$ means that there exists an edge $e$ of $\Gamma$ such that $w_2(e)\neq w_1(e)$. Then the invariant subspace $W$ also contains the element
\begin{align*}
v' &= \bigl(A^{2(w_2(e)+1)} + A^{-2(w_2(e)+1)}\bigr) v  + \rho\bigl( [\partial D_e] \bigr) (v) \\
&= \sum_{w\in \mathcal W_\Gamma} \bigl (A^{2(w_2(e)+1)} + A^{-2(w_2(e)+1)} -A^{2(w(e)+1)} - A^{-2(w(e)+1)}\bigr)  \alpha_w b_w
\end{align*}
Note that, in the basis $\mathcal B_\Gamma = \{ b_w\}_{w\in \mathcal W_\Gamma}$, the coordinate of $v'$ corresponding to $b_{w_1}$ is still non-zero since $w_2(e)\neq w_1(e)$, but that $v'$  now has one fewer non-zero coordinate than $v$ (because the coordinate of $v'$corresponding to $b_{w_2}$ is equal to  0). We can therefore replace $v$ by $v'\in W$, which is simpler. 

Iterating this construction, we eventually reach an element of $W$ that has exactly one non-zero coordinate, corresponding to $b_{w_1}$. This proves that $b_{w_1}$ belongs to $W$, as required.
\end{proof}

We will need to extend  Lemma~\ref{lem:InvSpaceContainsBasis} to a slightly more general framework. Let a \emph{partial spine} for the surface $\Sigma$ be the union $\Gamma$ of a finite family of disjoint trivalent graphs and simple closed curves in $\Sigma$, such that each component of $\Sigma - \Gamma$ contains at least one component of the boundary $\partial \Sigma$. This condition guarantees that $\Gamma$ can be enlarged to a trivalent spine $\hat \Gamma$ for $\Sigma$, by adding a few vertices and edges. 

The notion of an $N$--admissible weight system straightforwardly extends to partial spines: such a weight system consists of an $N$--admissible edge weight system on each trivalent graph component of $\Gamma$; and it assigns to each closed curve component $C$ of $\Gamma$ an even weight  $w(C)\in\{0,2,4, \dots, N-3\}$ if $N$ is odd, or a weight $w(C)\in\{0,1,2, \dots, \frac N2- 2 \}$ if $N$ is even. Again, plugging Jones-Wenzl idempotents into the edges and closed curve components of $\Gamma$ associates an element $\beta_w \in \SSS(H)$ to each $N$--admissible weight system $w\in \mathcal W_\Gamma$. As before,  we denote by $b_w=\Phi_H(\beta_w) \in V_S$  the image of $\beta_w$ under the map $\Phi_H \colon \SSS(H) \to V_S$ of Lemma~\ref{lem:WRTspaceSpannedBySkeins}, and we define $\mathcal B_\Gamma = \{b_w\}_{w\in \mathcal W_\Gamma}$. The only major difference is that $\mathcal B_\Gamma$ does not necessarily generate the Witten-Reshetikhin-Turaev space $V_S$.

\begin{lem}
\label{lem:InvSpaceContainsBasisPartialSpine}
The statement of Lemma~{\upshape\ref{lem:InvSpaceContainsBasis}} also holds when $\Gamma$ is only a partial spine for $\Sigma$. Namely, if $\Gamma$ is a partial spine for $\Sigma$ and if  $\sum_{w\in \mathcal W_\Gamma} \alpha_w b_w$ belongs to the invariant subspace $W$, then every  element $b_w \in \mathcal B_\Gamma$ with non-zero coefficient $\alpha_w\neq 0$ also belongs to $W$. 
\end{lem}

\begin{proof}
Enlarge the partial spine $\Gamma$ to a trivalent spine $\hat \Gamma$ for $\Sigma$, by adding vertices inside of the edges and closed curve components of $\Gamma$ and then adding edges connecting these new vertices as necessary. A weight system $w\in \mathcal W_\Gamma$  determines an $N$--admissible edge weight system for $\hat \Gamma$ as follows: it assigns to each edge of $\hat\Gamma$ that is contained in $\Gamma$ the $w$--weight of the edge or closed curve component of $\Gamma$ that contains it; and it assigns  weight 0 to each of the new edges of $\hat \Gamma - \Gamma$. This defines an inclusion $\mathcal W_\Gamma \subset \mathcal W_{\hat \Gamma}$, and we will use the same letter $w$ to denote the original  $w\in \mathcal W_\Gamma$ and the edge weight system $w\in \mathcal W_{\hat \Gamma}$ for $\hat \Gamma$  that it defines. 

We saw that a weight system $w\in \mathcal W_\Gamma$ for $\Gamma$ determines an element $\beta_w \in \SSS(H)$. Similarly, weighing the edges of $\hat \Gamma$ with $w \in \mathcal W_{\hat \Gamma}$ defines another element $\hat \beta_w \in \SSS(H)$. It turns out that $\hat \beta_w=\beta_w$. Indeed, this immediately follows from the idempotent property 
\SetLabels
\E( .3* .5) $a $ \\
\E( .72* .5) $a $ \\
( * ) $ $ \\
( * ) $ $ \\
( * ) $ $ \\
\endSetLabels
\raisebox{-5pt}{\AffixLabels{\includegraphics{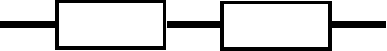}}}  
$=$ 
\SetLabels
\E( .5* .5) $a $ \\
( * ) $ $ \\
( * ) $ $ \\
( * ) $ $ \\
( * ) $ $ \\
\endSetLabels
\raisebox{-5pt}{\AffixLabels{\includegraphics{JW.pdf}}} 
 of Jones-Wenzl idempotents, which shows that adding vertices inside of the edges and closed curve components of $\Gamma$ does not change the associated element of $\SSS(H)$; it is immediate that adding weight 0 edges also has no impact. As a consequence, the inclusion $\mathcal W_\Gamma \subset \mathcal W_{\hat \Gamma}$ induces an inclusion $\mathcal B_\Gamma \subset \mathcal B_{\hat \Gamma} \subset V_S$. 

With this observation, every element $\sum_{w\in \mathcal W_\Gamma} \alpha_w b_w \in V_S $ can also be written as $\sum_{w\in \mathcal W_{\hat \Gamma}} \hat\alpha_w b_w$, by setting $\hat\alpha_w = \alpha_w$ when $w\in \mathcal W_{ \Gamma}\subset \mathcal W_{\hat\Gamma}$ and $\hat\alpha_w = 0$ when $w\in \mathcal W_{\hat \Gamma}-\mathcal W_\Gamma$. Lemma~\ref{lem:InvSpaceContainsBasisPartialSpine} then immediately follows by applying Lemma~\ref{lem:InvSpaceContainsBasis} to the trivalent spine $\hat \Gamma$. 
\end{proof}

 Lemma~\ref{lem:InvSpaceContainsBasis} shows that the invariant subspace $W$ contains at least one  element $b_w\in \mathcal B_\Gamma$ associated to a weight system $w\in \mathcal W_\Gamma$ for a trivalent spine $\Gamma$. Our next goal is to show that $W$ contains the vacuum element $b_0$ corresponding to the zero weight system $0\in \mathcal W_\Gamma$ for all partial spines $\Gamma$. The following definition is designed to  measure progress in this direction.

 The \emph{complexity} of a weight system $w\in \mathcal W_\Gamma$ for a partial spine $\Gamma$  is defined as the triple
 $$
  |w| = \bigl( e(\Gamma), \max(w), n_{\max}(w) \bigr) \in \N^3
 $$ 
 where $e(\Gamma)$ is the number of edges of $\Gamma$, $\max(w)$ is the largest weight assigned by $w$ to the edges and closed curve components of $\Gamma$, and where $n_{\max}(w)$ is the number of edges and closed curve components where this maximum is attained (and where $\N$ denotes the set of non-negative integers). We endow $\N^3$ with the lexicographic order.

\begin{lem}
\label{lem:DecreaseComplexity}
If the invariant subspace $W$ contains an element $b_w\in \mathcal B_\Gamma$ associated to a non-zero weight system $w\in \mathcal W_\Gamma$ for a partial spine $\Gamma$, then $W$ contains another element $b_{w'} \in \mathcal B_{\Gamma'}$, represented by a partial spine $\Gamma'$ and a weight system $w' \in \mathcal W_{\Gamma'}$, such that  $|w'|<|w|$. 
\end{lem}

\begin{proof}
We distinguish cases.

\medskip
\noindent\textsc{Case 1:} \emph{ The weight system $w$ assigns weight $0$ to an edge $e$ of $\Gamma$.}
The admissibility properties of $w$ imply that, if the endpoints of $e$ are not distinct and  correspond to the same vertex of $\Gamma$,  the third edge $e'$ emanating from this vertex has weight $w(e')$ equal to 0. Replacing $e$ by $e'$ if necessary, we can therefore assume that the endpoints of $e$ are distinct. The admissibility condition then shows that, at each of these endpoints, the two other adjacent edges have the same $w$--weight. Let $\Gamma'$ be the partial spine obtained from $\Gamma$ by removing $e$, and combining the edges that meet at each of its end vertices. By the above observation, $w$ induces a weight system $w'\in \mathcal W_{\Gamma'}$, and as in the proof of Lemma~\ref{lem:InvSpaceContainsBasisPartialSpine}, they are represented by the same basis element in $V_S$. By construction, $\Gamma'$ has one fewer edge than $\Gamma$, so that $|w'|<|w|$.

\begin{figure}[htbp]

\centerline{
$\displaystyle
\SetLabels
\E( .5* .5) \rotatebox{0}{\scalebox{.60}{$ w(e_0) $}} \\
( .135 *.85 )  \rotatebox{-60}{\scalebox{.60}{$ w(e_2) $}} \\
( .865* .625) \rotatebox{60}{\scalebox{.60}{$ w(e_3) $}} \\
( .13* .13)  \rotatebox{60}{\scalebox{.60}{$ w(e_1) $}} \\
( .87*.35 )  \rotatebox{-60}{\scalebox{.60}{$ w(e_4) $}}\\
\endSetLabels
\raisebox{-30pt}{\AffixLabels{{\includegraphics{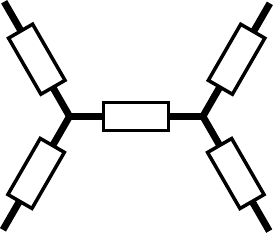}}}}
=  \quad  \sum_{w'} \  \begin{Bmatrix} w(e_1) &w(e_2) & w'(e_0') \\ w(e_3) &w(e_4) & w(e_0)  \end{Bmatrix} \ 
\SetLabels
( .5* .4) \rotatebox{90}{\scalebox{.55}{$ w'(e_0') $}} \\
( .25* .9)  \rotatebox{-30}{\scalebox{.60}{$ w(e_2) $}} \\
( .75*.8 ) \rotatebox{30}{\scalebox{.60}{$ w(e_3) $}} \\
( .26* .07)  \rotatebox{30}{\scalebox{.60}{$ w(e_1) $}} \\
( .75*.16 )  \rotatebox{-30}{\scalebox{.60}{$ w(e_4) $}} \\
\endSetLabels
\raisebox{-36pt}{\AffixLabels{\rotatebox{90}{\includegraphics{Flip.pdf}}}}
$
}
\caption{The Flip Relation}
\label{fig:Flip}
\end{figure}

\medskip
\noindent\textsc{Case 2:} \emph{No edge weight is $0$, and the maximum weight $\max(w)$ is attained on an edge $e_0$ of $\Gamma$ whose endpoints are distinct.}

We apply to $\Gamma$ the classical Flip Relation represented in Figure~\ref{fig:Flip}.  The Flip Move replaces $\Gamma$ by a new partial spine $\Gamma'$ that differs from $\Gamma$  only in the edge $e_0$, and replaces $e_0$ by an edge $e_0'$ that connects differently the four edges meeting $e_0$.  The Flip Relation 
$$
\beta_w = \sum_{w'} \begin{Bmatrix} w(e_1) &w(e_2) & w'(e_0') \\ w(e_3) &w(e_4) & w(e_0)  \end{Bmatrix} \beta_{w'} 
$$
expresses the element $\beta_w \in \SSS(H)$ represented by  $w\in \mathcal W_{\Gamma}$ as a linear combination of  elements $\beta_{w'}\in \SSS(H)$ where $w'$ ranges over all $N$--admissible weight systems for $\Gamma'$ that coincide with $w$ over all edges that are common to $\Gamma$ and $\Gamma'$. 

A key feature of this relation are the coefficients $ \begin{Bmatrix} w(e_1) &w(e_2) & w'(e_0') \\ w(e_3) &w(e_4) & w(e_0)  \end{Bmatrix}\in \C$, known as $6j$--symbols. A precise computation of these $6j$--symbols can be found in \cite{KauffmanLins} or \cite{MasbaumVogel}. The corresponding formula is usually complicated, and expresses a $6j$--symbol as a sum of several terms, each of which is a product of quantum integers and their inverses. However, it is somewhat simpler for the ``smallest'' of the $N$--admissible weight systems $w'$ that are compatible with $w$. 

Consider the weight system $w_1'$ that  coincides with $w$ on $\Gamma-e_0 = \Gamma'-e_0'$ and  assigns weight 
$$w_1'(e_0') = \max \bigl\{ | w(e_1) - w(e_4)|, |w(e_2) - w(e_3)|\bigr \}$$ 
to the edge $e_0'$. The formula is specially designed so that $w_1'$ is $N$--admissible. In fact, $w_1'$ is the edge weight that minimizes the weight $w'(e_0')$ among all $N$--admissible weight systems $w'\in \mathcal W_{\Gamma'}$ that coincide with $w$ outside of $e_0'$. We will not need this minimizing property, but  the following other feature of $w'$ is critical for our purposes:  for this specific weight system  $w_1'\in \mathcal W_{\Gamma'}$,  the sum occurring in the formula of \cite{MasbaumVogel} consists of a single term, and expresses the $6j$--symbol  $ \begin{Bmatrix} w(e_1) &w(e_2) & w_1'(e_0') \\ w(e_3) &w(e_4) & w(e_0)  \end{Bmatrix}$ as a product of non-zero quantum integers and their inverses. In particular, this $6j$--symbol is different from 0. 

Remembering that $b_w= \Phi_H(\beta_w)\in \mathcal B_{\Gamma}$ and $b_{w'} = \Phi_H(\beta_{w'})\in \mathcal B_{\Gamma'}$ for  the map $\Phi\colon \SSS(H) \to V_S$ of Lemma~\ref{lem:WRTspaceSpannedBySkeins}, 
$$
b_w = \sum_{w'} \begin{Bmatrix} w(e_1) &w(e_2) & w'(e_0') \\ w(e_3) &w(e_4) & w(e_0)  \end{Bmatrix} b_{w'} 
$$
in the Witten-Reshetikhin-Turaev space $V_S$. By hypothesis, $b_w$ belongs to the invariant subspace $W\subset V_S$. Lemma~\ref{lem:InvSpaceContainsBasisPartialSpine}  then shows that $W$ also contains the element $b_{w_1'} \in \mathcal B_{\Gamma'}$ corresponding to   $w_1'\in \mathcal W_{\Gamma'}$, since its coefficient in the above sum is different from 0.

By our hypothesis that the weights assigned by $w$  to the edges of $\Gamma$ are non-zero and bounded by $w(e_0)=\max(w)$, the weight $w_1'(e_0') $ defined above is strictly less than $\max(w)$. It follows that $w_1'$ has lower complexity $|w_1'|<|w|$ than $w$. 

\medskip
\noindent\textsc{Case 3:} \emph{No edge weight is $0$, and the maximum weight $\max(w)$ is attained on an edge $e_0$ of $\Gamma$ whose endpoints are identified.}

Because the endpoints of $e_0$ are identified, we cannot apply a Flip Move at $e_0$. Instead, we will apply such a move at the remaining edge $e_1$ that is adjacent to the vertex corresponding to the two ends of  $e_0$. This gives a new partial spine $\Gamma'$, obtained from $\Gamma$ by replacing the edge $e_1$ by an edge $e_1'$ as in Figure~\ref{fig:Flip2}. 

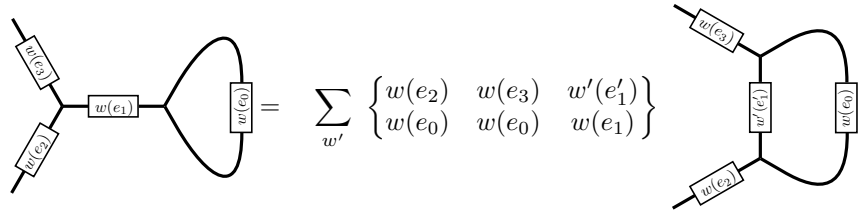
\begin{figure}[htbp]

\centerline{
$\displaystyle
\SetLabels
\E( .43* .49) \rotatebox{0}{\scalebox{.60}{$ w(e_1) $}} \\
( .115 *.83)  \rotatebox{-60}{\scalebox{.60}{$ w(e_3) $}} \\
( .11* .13) \rotatebox{60}{\scalebox{.60}{$ w(e_2) $}} \\
\E( .955* .5)  \rotatebox{90}{\scalebox{.60}{$ w(e_0) $}} \\
\endSetLabels
\raisebox{-30pt}{\AffixLabels{{\includegraphics{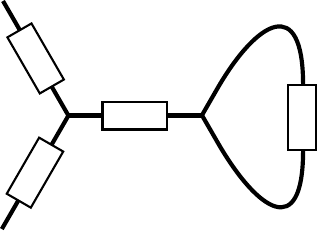}}}}
=  \quad  \sum_{w'} \  \begin{Bmatrix} w(e_2) &w(e_3) & w'(e_1') \\ w(e_0) &w(e_0) & w(e_1)  \end{Bmatrix} \ 
\SetLabels
( .475* .4) \rotatebox{90}{\scalebox{.55}{$ w'(e_1') $}} \\
( .935* .4)  \rotatebox{90}{\scalebox{.60}{$ w(e_0) $}} \\
( .24*.9 ) \rotatebox{-30}{\scalebox{.60}{$ w(e_3) $}} \\
( .25* .06)  \rotatebox{30}{\scalebox{.60}{$ w(e_2) $}} \\
\endSetLabels
\raisebox{-36pt}{\AffixLabels{\rotatebox{0}{\includegraphics{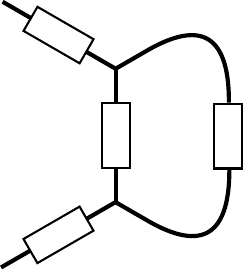}}}}
$
}
\caption{A special case of the Flip Relation}
\label{fig:Flip2}
\end{figure}

We then consider for $\Gamma'$ the $N$--admissible weight system $w_1'$ that assigns weight 
$$w_1'(e_1') = \max \bigl \{ w(e_0)-w(e_2), w(e_0)-w(e_3)\bigr  \} $$
to the edge $e_1'$ and coincides with $w$ outside of $e_1'$. As in Case~2, the formula of \cite{MasbaumVogel} shows that the $6j$--symbol occurring as coefficient of $\beta_{w_1'}$ in the Flip Relation of Figure~\ref{fig:Flip2} is non-zero. An application of Lemma~\ref{lem:InvSpaceContainsBasisPartialSpine} again proves that the invariant subspace $W\subset V_S$ contains the element $b_{w_1'}\in \mathcal B_{\Gamma'}$ associated to $w_1' \in \mathcal W_{\Gamma'}$.

Note that $w_1'(e_1') < w(e_0) = \max (w)$, so that $|w_1'| \leq |w|$. Because the inequality is not necessarily strict, we are not quite done yet. However, we can now apply a Flip Move to $\Gamma'$ at the edge $e_0$, and use Case~2  to conclude.

\medskip
\noindent\textsc{Case 4:} \emph{$N$ is odd, and the maximum weight $\max(w)$ is attained on a closed curve component $C$ of $\Gamma$.}

Push this simple closed curve $C\subset \Sigma$ to  $\Sigma \times\{1\} \subset \partial H =S$ to consider it as a partial spine in a thickening of $S$, and let $[C^{S_2}] \in \SSS(S)$ be defined by assigning weight 2 to this partial spine, namely by plugging a Jones-Wenzl idempotent    \SetLabels
\E( .5* .5) $2 $ \\
( * ) $ $ \\
( * ) $ $ \\
( * ) $ $ \\
( * ) $ $ \\
\endSetLabels
\raisebox{-5pt}{\AffixLabels{\includegraphics{JW.pdf}}} 
in $C$; see Remark~\ref{rem:ChebyWRT}
 to explain the notation.

\begin{figure}[htbp]

$\SetLabels
( .5* .3) \tiny$ a$ \\
(.5 * .08) \tiny $ 2$ \\
( * ) $ $ \\
( * ) $ $ \\
( * ) $ $ \\
\endSetLabels
\raisebox{-17pt}{\AffixLabels{\includegraphics{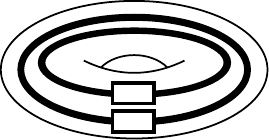}}}
=
\SetLabels
( .5* .19) \tiny$ a\kern -3pt -\kern -3pt 2$ \\
( * ) $ $ \\
( * ) $ $ \\
( * ) $ $ \\
\endSetLabels
\raisebox{-17pt}{\AffixLabels{\includegraphics{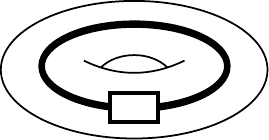}}}
+
\SetLabels
( .5* .19) \tiny$ a$ \\
( * ) $ $ \\
( * ) $ $ \\
( * ) $ $ \\
\endSetLabels
\raisebox{-17pt}{\AffixLabels{\includegraphics{TorusJW.pdf}}}
+
\SetLabels
( .5* .19) \tiny$ a\kern -3pt +\kern -3pt 2$ \\
( * ) $ $ \\
( * ) $ $ \\
( * ) $ $ \\
\endSetLabels
\raisebox{-17pt}{\AffixLabels{\includegraphics{TorusJW.pdf}}}
$
\caption{A multiplication property for Jones-Wenzl idempotents}
\label{fig:TorusJW}
\end{figure}

If the subspace $W\subset V_S$ contains $b_w \in \mathcal B_\Gamma$, it also contains $\rho\bigl( [C^{S_2}] \bigr)(b_w)$ by invariance of $W$ under the action of $\rho\bigl( \SSS(S) \bigr)\subset \End(V_S)$. The relation of Figure~\ref{fig:TorusJW}, valid in a solid torus neighborhood of $C$, enables us to compute this element and gives 
$$
\rho\bigl( [C^{S_2}] \bigr)(b_w) = b_{w'} + b_w +  b_{w''}
$$
where the weight systems $w'$ and $w''$ for $\Gamma$ coincide with $w$ outside of the closed curve component $C$, and respectively assign weight $w(C)-2$ and $w(C) +2$ to $C$. See \cite[Lemma~14.11]{LickorishBook} for a proof of this relation. 

There is a little caveat needed here when $w(C)$ is equal to its maximum possible value $N-3$. Then, $w''$ assigns weight $N-1$ to $C$, and consequently is no longer $N$--admissible. The associated element $\beta_{w''} \in \SSS(H)$ still makes sense, but its image $\Phi_H (\beta_w)\in V_S$ under the map of Lemma~\ref{lem:WRTspaceSpannedBySkeins} is equal to 0. We consequently set $b_{w''}=0$ in this case. 

In all cases, we can apply Lemma~\ref{lem:InvSpaceContainsBasisPartialSpine} to $\rho\bigl( [C^{S_2}] \bigr)(b_w) \in W$, and we conclude that $b_{w'}$ belongs to $W$. By construction, $w'(C) < w(C) = \max(w)$, so that $w'\in \mathcal W_\gamma$ has lower complexity $|w'|<|w|$. This concludes the proof in this case.

\medskip
\noindent\textsc{Case 5:} \emph{$N$ is even, and the maximum weight $\max(w)$ is attained on a closed curve component $C$ of $\Gamma$.}

The proof is almost identical to that of Case~4, except that we do not have to worry about keeping all weights even. Because of this, it suffices to consider the action of the element $[C]\in \SSS(S)$ represented by the closed curve $C$. Then, a computation similar to that of  Figure~\ref{fig:TorusJW} (see again \cite[Lemma~14.11]{LickorishBook}) gives that
$$
\rho\bigl( [C] \bigr)(b_w) = b_{w'} + b_w +  b_{w''}
$$
where the weight systems $w'$ and $w''$ for $\Gamma$ coincide with $w$ outside of the closed curve component $C$, and respectively assign weight $w(C)-1$ and $w(C)+1$ to $C$ (with $b_{w''}=0$ when $w(C) = \frac N2-2$). An application of Lemma~\ref{lem:InvSpaceContainsBasisPartialSpine} then shows that $W$ contains the basis element $b_{w' } \in \mathcal B_\Gamma$. Again, $w'\in \mathcal W_\gamma$ has lower complexity $|w'|<|w|$, and this concludes the proof in this case. 

Since the five cases considered exhaust all possibilities, the proof of Lemma~\ref{lem:DecreaseComplexity} is now complete. 
\end{proof}

We are now almost done with the proof of Theorem~\ref{thm:WRTirred}. 

Recursively applying Lemma~\ref{lem:DecreaseComplexity}, we eventually reach a partial spine $\Gamma'$ such that the invariant subspace $W$ contains the element $b_0 \in \mathcal B_{\Gamma'}$ associated to the trivial weight system $0 \in \mathcal W_{\Gamma'}$. By definition $b_0$ is also the image of the empty skein $[\varnothing] \in \SSS(H)$ under the map $\Phi_H$ of Lemma~\ref{lem:WRTspaceSpannedBySkeins}. 

Let $L$ be a framed link in the handlebody $H$. Push $L$ into a tubular neighborhood of the boundary $\partial H = S$, so that $L$ defines a skein $[L] \in \SSS(S)$. 

By invariance of $W$ under the under the action of $\rho\bigl( \SSS(S) \bigr)\subset \End(V_S)$, it contains the element $\rho\bigl( [L] \bigr)(b_0)$. However, we also have, from the definition of the Witten-Reshetikhin-Turaev homomorphism $\rho$, that
$$
\rho\bigl( [L] \bigr)(b_0) = \rho\bigl( [L] \bigr) \bigl( \Phi_H([\varnothing]) \bigr) = \Phi_H\bigl([L\cup \varnothing] \bigr) =\Phi_H\bigl([L] \bigr) 
$$
where, in the first two terms, $[L]$ denotes the element of $\SSS(S)$ represented by $L$, whereas $[L]$ is the element of $\SSS(H)$ represented by $L$ in the last two terms.

This proves that the invariant subspace $W\subset V_S$ contains the image $\Phi_H\bigl([L] \bigr) $ of every skein $[L] \in \SSS(H)$. Since these skeins generate $\SSS(H)$ and since $\Phi_H \colon \SSS(H) \to V_S$ is surjective by Lemma~\ref{lem:WRTspaceSpannedBySkeins}, this proves that $W$ is equal to the whole space $V_S$. 

This completes the proof of Theorem~\ref{thm:WRTirred}. 
\end{proof}

\section{The classical shadow of the Witten-Reshetikhin-Turaev representation}

\subsection{Threading polynomials along a framed link}

Given a framed (connected) knot $K$ in a 3--dimensional manifold $M$ and  a polynomial
$
P(x) = \sum_{i=0}^n a_i x^i
$, we can consider  the linear combination
$$
[K^P] = \sum_{i=0}^n a_i\, [K^{(i)}] \in \SSS(M)
$$
where, for each $i$, $K^{(i)}$ is the framed link obtained by taking $i$ parallel copies of $K$ in the direction indicated by the framing.
More generally, if $L\subset M$ is a framed link with components $K_1$, $K_2$, \dots, $K_l$, define
$$
[L^P] = \sum_{0\leq i_1, \,i_2,\dots, \,i_l \leq n} a_{i_1} a_{i_2} \dots a_{i_l}\, [K_1^{(i_1)} \cup K_2^{(i_2)} \cup \dots \cup K_l^{(i_l)} ] \in \SSS(M).
$$
By definition,  $[L^P]\in \SSS(M)$ is obtained by \emph{threading the polynomial $P$ along the framed link $L$}. 

We will apply this construction to the (normalized) Chebyshev polynomials of the first and second type. 

The \emph{$n$--th Chebyshev polynomial of the first type} $T_n(x)$ is defined by the properties that $T_n(x) = xT_{n-1}(x) - T_{n-2}(x)$, $T_0(x) =2$ and $T_1(x) =x$. The \emph{$n$--th Chebyshev polynomial of the second type} $S_n(x)$ is defined by the same recurrence relation $S_n(x) = xS_{n-1}(x) - S_{n-2}(x)$, the same  initial condition $S_1(x) =x$, but differs in the other initial condition $S_0(x) =1$. The two types of Chebyshev polynomials are related by the property that $T_n(x) = S_n(x) - S_{n-2}(x)$ for every $n$. 

\begin{rem}
\label{rem:ChebyWRT}
The Chebyshev polynomials of the second type $S_n(x)$  are ubiquitous in the Witten-Reshetikhin-Turaev theory and, more generally, in the representation theory of the quantum group $\mathrm U_q(\mathfrak{sl}_2)$. In particular,  for each framed link $L$ in a 3--manifold $M$,  the element of $\SSS(M)$ obtained by plugging the $n$--th Jones-Wenzl idempotent in each component of $L$ is equal to the element $[L^{S_n}] \in \SSS(M)$. See  \cite[\S 13]{LickorishBook} or \cite[p.~715]{LickorishHandbook}. 
\end{rem}

The following facts, which can for instance be found in Lemma 6.3 of \cite{BHMV3man}, are  crucial for our computations. 

\begin{lem}
\label{lem:ChebyRelations}
Suppose that $A$ is a primitive $2N$--root of unity with $N$ odd, and let $V_S$ be the Witten-Reshetikhin-Turaev space of the surface $S$. 
Let $K$ and $L$ be two disjoint framed links in a $3$--manifold $M$ bounded by $S$. Then, for the homomorphism $\Phi_M \colon \SSS(M ) \to V_S$ of Lemma~{\upshape\ref{lem:WRTspaceSpannedBySkeins}}, 
\begin{enumerate}
\item $\Phi_M \bigl( [K^{S_{N-1}} \cup L] \bigr) = 0$;
\item $\Phi_M \bigl( [K^{S_{N-2-n}} \cup L] \bigr)  = \Phi_M \bigl( [K^{S_{n}} \cup L] \bigr) $ for every integer $n$. \qed
\end{enumerate}
\end{lem}

\subsection{The classical shadow of the Witten-Reshetikhin-Turaev representation}
As usual, $S$ is a connected closed oriented surface. 
Consider the character variety
$$
\RR_{\SL(\C)}(S) = \{ \text{homomorphisms } r \colon \pi_1(S) \to \SL(\C)\}\db \SL(\C)
$$
where $\SL(\C)$ acts on homomorphisms $\pi_1(S) \to \SL(\C)$ by conjugation, and where the double bar indicates that one takes the quotient in the sense of geometric invariant theory. In practice, this means that two homomorphisms $r$, $r'\colon \pi_1(S) \to \SL(\C)$ represent the same point of $\RR_{\SL(\C)}(S)$ if and only if they induce the same trace functions, namely if and only if $\Tr\,r(\gamma) = \Tr\,r'(\gamma) $ for every $\gamma \in \pi_1(S)$. 


\begin{thm}[\cite{BonWonSkeinReps1}]
\label{thm:ClassicShadow}
Suppose that $A$ is a primitive $2N$--root of unity with $N$ odd. If $\rho\colon \SSS(S) \to \End(V)$ is an irreducible representation of the skein algebra $\SSS(S)$, then there exists a unique character $r _\rho\in \RR_{\SL(\C)}(S)$ such that
$$
\rho \bigl( [K^{T_N}] \bigr) = - \Tr\,r_\rho(K)\,\Id_V
$$
for every framed knot $K \subset S \times [0,1]$. \qed
\end{thm}

By definition, $r _\rho\in \RR_{\SL(\C)}(S)$ is the \emph{classical shadow} of the representation $\rho\colon \SSS(S) \to \End(V)$. See also \cite{Le} for an alternative approach to the key properties underlying this statement. 

\begin{thm}
\label{thm:ClassiShadowWRT}
When $A$ is a primitive $2N$--root of unity with $N$ odd, the classical shadow of the Witten-Reshetikhin-Turaev representation $\rho\colon \SSS(S) \to \End(V_S)$ is the trivial character $\iota \in  \RR_{\SL(\C)}(S)$, represented by the trivial homomorphism $\pi_1(S) \to \SL(\C)$. 
\end{thm}

\begin{proof}
This is a relatively simple consequence of Lemma~\ref{lem:ChebyRelations}. Identify $S\times [0,1]$ to a tubular neighborhood of the boundary $S=\partial M$ in the 3--manifold $M$. To compute $\rho \bigl( [K^{T_N}] \bigr) \in \End(V_S)$ for a framed knot $K \subset S \times [0,1]$, Lemma~\ref{lem:WRTspaceSpannedBySkeins} shows that it suffices to consider its action on those elements of $V_S$ of the form $v= \Phi_M \bigl([L] \bigr) $  for a framed link $L\subset M$. Pushing $L$ away from the neighborhood  $S\times [0,1]$ of $S=\partial M$ in $M$,
\begin{align*}
\rho \bigl( [K^{T_N}] \bigr) (v)& =  \Phi_M\bigl( [K^{T_N} \cup L] \bigr) =    \Phi_M\bigl( [K^{xS_{N-1} - 2S_{N-2}} \cup L] \bigr)  \\
&=  \Phi_M\bigl( [K^{xS_{N-1}} \cup L] \bigr)  -  2\Phi_M\bigl( [K^{S_{N-2}} \cup L] \bigr),
\end{align*}
using the property that $T_n(x) = S_n(x) - S_{n-2}(x) = xS_{n-1}(x) -2S_{n-2}(x)$ for every $n$. 

The term $[K^{xS_{N-1}} \cup L] \in \SSS(M)$ is also equal to $[K^{S_{N-1}}\cup K' \cup L] $ where $K'$ is a push-off of $K$ in the direction given by the framing. Its image $\Phi_M \bigl(  [K^{S_{N-1}}\cup K' \cup L]  \bigr)$ in $V_S$ is therefore equal to 0 by Part~(1) of Lemma~\ref{lem:ChebyRelations}. Similarly, Part~(2) of Lemma~\ref{lem:ChebyRelations} shows that
$$
\Phi_M\bigl( [K^{S_{N-2}} \cup L] \bigr) = \Phi_M\bigl( [K^{S_{0}} \cup L] \bigr) = \Phi_M\bigl( [K^{1} \cup L] \bigr)= \Phi_M\bigl( [ L] \bigr) =v.
$$
(Note that, since $1=x^0$, the skein $[K^1]=[K^{(0)}]$ is represented by 0 copies of the knot $K$, and is therefore trivial.)

Therefore, $\rho \bigl( [K^{T_N}] \bigr) (v) = -2v$ for every $v =  \Phi_M\bigl( [ L] \bigr)  \in V_S$ represented by a framed link $L \subset M$. Since these elements generate $V_S$ by Lemma~\ref{lem:WRTspaceSpannedBySkeins}, it follows that $\rho \bigl( [K^{T_N}] \bigr)=-2\,\Id_{V_S}$. 

If $r_\rho$ is the classical shadow of the Witten-Reshetikhin-Turaev representation $\rho\colon \SSS(S) \to \End(V_S)$, this proves that $\Tr\, r_\rho(K)=2$ for every knot $K \subset S\times[0,1]$. This means that $r _\rho\in \RR_{\SL(\C)}(S)$ is the character represented by the trivial homomorphism. 
\end{proof}

\bibliographystyle{amsplain}

 \bibliography{WRTrep}

\end{document}